\documentclass{amsart}
\usepackage{amssymb}
\usepackage{psfig}
\newtheorem{thm}{Theorem}
\newtheorem{crl}{Corollary}
\newtheorem{prop}{Proposition}
\newtheorem{lem}{Lemma}
\newtheorem{cnj}{Conjecture}
\newtheorem{defn}{Definition}
\newtheorem{ex}{Example}
\newtheorem{rem}{Remark}
\newcommand{\bin}[2]{\genfrac[]{0pt}1{#1}{#2}^{\vphantom{A}}}
\newcommand{\sdp}[2]{\Sigma_{#1}({#2})}
\title{Weight Multiplicity Polynomials  \\ of multi-variable Weyl Modules}
\author{S.Loktev}
\address{Institute for Theoretical and Experimental Physics, Moscow 117218 Russia}
\email{loktev@itep.ru}
\dedicatory{To Pierre Deligne on the occasion of his 65th birthday}
\begin{document}
\maketitle

\def\theenumi{\roman{enumi}}
\def\labelenumi{(\theenumi)}

\def \ch {{\rm ch}}
\def \gr {{\rm gr}\,}
\def \g  {{\mathfrak g}}
\def \b  {{\mathfrak b}}
\def \h  {{\mathfrak h}}
\def \a  {{\mathfrak a}}
\def \rkg {{{\rm rk}(\g)}}
\def \gs {{\mathfrak sl}}
\def \gl {{\mathfrak gl}}
\def \CC {{\mathbb C}}
\def \QQ {{\mathbb Q}}
\def \NN {{\mathbb N}}
\def \cpf {\CC{\rm PF}}
\def \da {{\mathcal A}}
\def \prt {\partial/\partial t}
\def \sgr {{\rm Sign}}
\def \v {{\bf m}}
\def \ve {\varepsilon}

\begin{abstract}
This paper is based on the observation that dimension of weight spaces of multi-variable Weyl modules
depends polynomially on the highest weight (Conjecture~\ref{cmain}). We support this conjecture
by various explicit answers for up to three variable cases and discuss the underlying combinatorics.
\end{abstract}

\section*{Introduction}

Representations of one-variable current algebras are now well studied and have a lot of applications
to algebraic geometry and mathematical physics. Also there are several approaches
to develop a similar theory for multi-variable currents. Let us briefly outline some  of them.

The first one is an affinization of affine Kac-Moody algebras (that is studying
Laurent polynomials in one variable with values in an affine Lie algebra). 
It provides a lot of structures, in particular,
the universal central extension and a canonical way to quantize such algebras, 
but makes sense only for the two-variable case.

The second approach starts from multi-variable, usually toroidal settings (that is studying
Laurent polynomials in several variables with values in a finite-dimensional Lie algebra). It is more universal,
but very few known notions and methods works in this generality, so other tools should
be used (see e.g. \cite{MRY}).

The third one is adding a variable using the fusion product introduced in \cite{FL1}. Namely,
for a set of representations of a Lie algebra $\a$ it produces a graded action 
of the current algebra $\a \otimes \CC[t]$ on their tensor product (see \cite{FL1}, \cite{FF1} how it
works for $\a = sl_2$ and \cite{FF2} for $\a = \widehat{sl_2})$.

Recently we were faced with a rather simply stated question which appeared to be related to fusion product 
in the one variable case, related to double affine algebras for two variables, promising for three variables and
pretty challenging by now for a higher number.

Suppose we have even the simplest Lie algebra $sl_2$ with the standard basis $e$, $h$, $f$. Consider
the Lie algebra of $sl_2$-valued polynomials $sl_2 \otimes \CC[x^1, \dots, x^d]$.
Let us follow the most naive notion of highest weight vector, namely let $e \otimes P$
acts on it by zero, and  $h\otimes P$ acts by the scalar $nP(0)$ for each polynomial $P$.
In order to obtain a finite-dimensional module let $n$ be a non-negative integer.
Following \cite{CP} introduce $W_{n}$ to be the module, 
generated by a highest weight vector $v$ with the additional integrability
relation $(f\otimes 1)^{n+1}v=0$. Then it appears (see \cite{FL2}) that $W_{n}$ is finite-dimensional
and that there are some pretty formulas for its dimension and dimensions of weight spaces.

The next table is organized as follows. First column indicates the number of variables.
Second column contains the dimensions of Weyl modules together with the number of
the sequence in the On-Line Encyclopedia of Integer Sequences \cite{Sl}. In the third column
there are the dimensions of weight spaces organized as polynomials. Last column
provides a reference.

\medskip

\noindent
\begin{tabular}{|c|c|c|c|}
\hline
$d$ & $\dim W_{n}$ & $\dim W_{n}^{n\omega - k \alpha} = P(n)/P(k)$ & reference \cr
\hline\hline
$0$ & $n+1$ & $1$ for $0\le k \le n$ & well known \cr 
\hline
$1$ & $2^n$ & 
{\bf Binomial Coefficients} $\bin{n}{k}$ & \cite{CP2} \cr
&
& 
$P(n) = \prod\limits_{i=0}^{k-1}(n-i)$ &  \cr
\hline
$2$ & {\bf Catalan Number}  & {\bf Narayana Numbers} & \cite{FL2}\cr
& $ \left.\bin{2n+2}{n}\right/ (n+1)$ & 
$\left. \bin{n+1}{k}\bin{n+1}{n-k}\right/ (n+1)$ & using  \cr
& A000108 
& $P(n) = \prod\limits_{i=-1}^{k-2}(n-i)\prod\limits_{i=0}^{k-1}(n-i)$ & 
\cite{H} \cr
\hline
$3$ & $\left.\bin{3n+3}{n}\right/\bin{n+2}{2}$ & $\frac{\bin{n+k+2}{2k+1}\bin{2n-k+1}{k}}{(k+1)(n+k+2)}^{\vphantom{s}}$
 & Conjec- \cr
& A000139 \rule[-0.7cm]{0cm}{1.6cm} & $P(n) = \prod\limits_{i=2}^{2k+1}(n+i-k) \prod\limits_{i=-1}^{k-2} (2n-i-k)$ & 
 ture~\ref{triag}\cr
\hline
$\ge 4$  & Huge prime factors & $P(n)$ is not factorizable /${\mathbb Q}$ & not known \cr
\hline
\end{tabular}

\bigskip

Note that in this table for a fixed $k$ the dimensions of $W_{n}^{n\omega - k \alpha}$ 
are values of a polynomial in $n$. It seems to be a general phenomenon and the paper is 
build around this polynomiality conjecture. In Section 1 we introduce the notation, discuss Weyl modules in full generality and state the polynomiality conjecture (Conjecture~\ref{cmain}). 
In Section 2 we produce several explicit answers and
support this conjecture in various particular cases (up to three variables). 

{\bf Acknowledgements.}
 It is a pleasure to dedicate this paper to the 65-th birthday of Pierre Deligne, who discovered
many wonderful results, in particular, a polynomial extrapolation of representations \cite{D1},
\cite{D}.
I am very grateful to A.Berenstein, P.Etingof, B.Feigin, V.Ostrik, A.Postnikov
for very useful and stimulating discussions. Part of this work
was done during a visit to MSRI and UC Riverside. The author was partially
supported by RF President Grant N.Sh-3035.2008.2, grants RFBR-08-02-00287, RFBR-CNRS-07-01-92214, 
RFBR-IND-08-01-91300 and P.Deligne 2004 Balzan prize in mathematics.

\section{Multi-variable Weyl modules}

\subsection{Notation}

Let $\g$ be a reductive Lie algebra. By $\b$ and $\h$ denote its Borel and Cartan subalgebras, by
$R$ and $R_+$ denote the sets of roots and positive roots, by $\alpha_i$, $i=1\dots \rkg$, denote simple roots,
by $\omega_i$ denote simple weights,
by $Q$ and  $Q^+$ denote the root lattice and its positive cone, by $P$ and $P^+$ denote
the weight lattice and the dominant weight cone. 

For $\g = \gl_r$ by $\ve_i$, $i=1\dots r$, denote
the standard basis in the weight space. Then partitions
$\xi = (\xi_1 \ge \dots \ge \xi_r)$ correspond to 
dominant weights $\sum \xi_i \ve_i$ of  $\gl_r$. 
For each $\xi$ by $\xi^t$ denote the {\em transposed} partition
$\xi^t_j = |\{i | \xi_i \ge j\}|$ corresponding to the reflected Young diagram.

For a representation $U$ by $U^\mu$ denote the corresponding weight space, that is the common
eigenspace of $\h$ where $h \in \h$ acts by the scalar $\mu(h)$.

In addition let $A$ be a commutative finitely-generated algebra with a unit $1$ and 
a co-unit (augmentation) $\epsilon: A \to \CC$. 
Note that $A$ can be treated as the algebra of functions on an affine scheme $M$, 
so $\epsilon$ is the evaluation at a certain closed point $p$ of $M$.

By $A_\epsilon$ denote the augmentation ideal, that is, the kernel of $\epsilon$. 
Note that the infinitesimal neighborhood spaces $A/A_\epsilon^n$ are always finite-dimensional.

\subsection{Highest weight modules}

Here we study finite-dimensional representations of $\g \otimes A$, that is, the Lie algebra
of $\g$-valued functions on the scheme $M$.

\begin{defn}
Let $U$ be a representation of $\g \otimes A$.
We say that a vector $v_{\lambda} \in U$ is a {\em highest weight vector}
of weight $\lambda \in \h^*$ if
$$(g \otimes P) v_{\lambda} = \lambda(g) \epsilon(P) v_{\lambda}\quad \mbox{for}\ g \in \b,\ P \in A.$$
\end{defn}

\begin{thm}\cite{FL2}\label{wex}
\begin{enumerate}
\item There exists a universal finite-dimensional module $W^{A}_\epsilon(\lambda)$, 
such that any finite-dimensional module
generated by $ v_{\lambda}$ is a quotient of $W^A_\epsilon(\lambda)$.

\item We have  $W^A_\epsilon(\lambda)\ne 0$ if and only if $\lambda \in P^+$.

\item For any $\lambda$ there exists $N$ such that $\g \otimes A_\epsilon^N$ acts on
$W^A_\epsilon(\lambda)$ by zero.

\item We have $W^A_\epsilon(\lambda) \cong  \oplus_\mu W^A_\epsilon(\lambda)^\mu$, and 
$W^A_\epsilon(\lambda)^\mu \ne 0$ if and only if
$w(\mu) - \lambda \in Q^+$ for any element $w$ of the Weyl group.

\item Any finite-dimensional module, generated by a common eigenvector of $\b \otimes A$,
is a quotient of $W^A_{\epsilon_1}(\lambda^1) \otimes \dots \otimes W^A_{\epsilon_k}(\lambda^k)$ 
for some $\lambda^i$, $\epsilon_i$.
\end{enumerate}
\end{thm}

\begin{defn}\cite{CP}\cite{FL2}
The module $W^A_\epsilon(\lambda)$ is called {\em Weyl module}. For $A= \CC[x^1, \dots, x^d]$ and
$\epsilon(P) = P(0)$ let us denote it by $W^d(\lambda)$ as well.
\end{defn}

\begin{prop}
\begin{enumerate}
\item For each homomorphism $f: A_1 \to A_2$ we have a natural map 
$f^*: W^{A_1}_{\epsilon \circ f}(\lambda) \to W^{A_2}_{\epsilon}(\lambda)$.
\item If $\epsilon$ corresponds to a non-singular point of $M$ then we have
$W^A_\epsilon(\lambda) \cong W^d(\lambda)$, where $d=\dim M$.
\end{enumerate}
\end{prop}

\begin{proof}
For (i) note that due to the map $f$ there is an action of $\g \otimes A_1$ on $W^{A_2}_{\epsilon}(\lambda)$,
and the $\g \otimes A_1$-submodule, generated by the highest weight vector, is a quotient
of $W^{A_1}_{\epsilon \circ f}(\lambda)$.

For (ii) note that for
$\epsilon$ corresponding to a non-singular point
the quotient $A/ A_\epsilon^N$ is isomorphic to the same factor for $\CC[x^1, \dots, x^d]$, so
by Theorem~\ref{wex}~(iii) the Weyl modules are isomorphic.
\end{proof}

In other words, characters of Weyl modules can be considered as a functorial invariant of algebraic singularities.

\begin{cnj}\label{cmain}
Let $\lambda = \sum \lambda_i \omega_i$. Fix $\mu = \sum \mu^i \alpha_i$.
\begin{enumerate}
\item
For $d>0$ we have $\dim W^d(\lambda)^{\lambda-\mu}$
is a polynomial in $\lambda_1, \dots, \lambda_\rkg$ of degree $\mu^i d$ in the variable $\lambda_i$,
$i=1\dots \rkg$.
\item 
Even for a singular point there exists $\nu$, such that for $\lambda - \nu \in P^+$ we have 
$\dim W^A_\epsilon(\lambda)^{\lambda-\mu}$
is a polynomial in $\lambda_1, \dots, \lambda_\rkg$ of degree $\mu^i d$ in the variable $\lambda_i$,
$i=1\dots \rkg$, where $d=\dim M$.
\end{enumerate}
\end{cnj}

\begin{ex} Let $\g = sl_2$, $\mu = \alpha$. Then $W^A_\epsilon(n\omega)^{n\omega-\alpha} \cong A / A_\epsilon^n$ 
(see \cite{FL2}). 
And for a big enough $n$ the integers  
$\dim A / A_\epsilon^n - \dim A/ A_\epsilon^{n+1} = \dim  A_\epsilon^n/ A_\epsilon^{n+1}$
are values of the Hilbert polynomial for the completion
of $A$ at $\epsilon$.
\end{ex}

In this paper we support and refine this conjecture in several particular cases.

\subsection{Schur-Weyl duality}

First let us describe the classical construction.
Let $V$ be the $r$-dimensional vector representation of
$\gl_r$. For the classical settings note that $V^{\otimes n}$ inherits both the
action of the Lie algebra $\gl_r$ and the symmetric group $\Sigma_n$, and that
these actions commute. Decomposition of this bi-module provides a reciprocity 
between the representations of these objects and it can be described as the 
following functor.

\begin{defn}
Let $\pi$ be a representation of $\Sigma_n$. By $\nabla_n^r(\pi)$ denote 
the representation $\left(V^{\otimes n}\otimes \pi\right)^{\Sigma_n}$ of $\gl_r$.
\end{defn}

\begin{rem}
Note that at the level of characters $\nabla_n^r$ is known as the {\em Frobenius
characteristic map}.
\end{rem}

This functor can be generalized for our settings as follows. 

\begin{defn}
For an associative algebra $A$ and an integer $n$ by 
{\em wreath product} $\sdp{n}{A}$
denote the associative algebra, generated by $A^{\otimes n}$
and the elements of $\Sigma_n$ under the relation
\begin{equation}\label{comm}
(a_1\otimes \dots\otimes a_n) \cdot \sigma = \sigma \cdot (a_{\sigma(1)}\otimes \dots
\otimes a_{\sigma(n)})
\end{equation}
and identification of the unit $e \in \Sigma_n$ with $1\otimes \dots \otimes 1 
\in A^{\otimes n}$. 
\end{defn}

In other words, we have $\sdp{n}{A}\cong \CC[\Sigma_n] \otimes A^{\otimes n}$
as a vector space with the multiplication given by the relation~\eqref{comm}.
Let us show that $\nabla_n^r(\pi)$ is indeed a functor from representations of 
$\sdp{n}{A}$ to the representations
of the Lie algebra $\gl_r \otimes A$. 
Note that for this observation $A$ is not necessarily commutative, and  $\gl_r \otimes A$
can be defined as
the space of matrices with entries from $A$ with the usual commutator.

\begin{prop}
Let $\pi$ be a representation of $\sdp{n}{A}$. Then $\gl_r \otimes A$
acts on 
$\nabla_n^r(\pi)$
as follows
\begin{equation}\label{swact}
(g \otimes P) (u_1\otimes \dots \otimes u_n \otimes v) = \sum_{k=1}^n 
u_1 \otimes \dots \otimes g(u_k) \otimes \dots \otimes u_n \otimes \iota_k (P) v,
\end{equation}
where $\iota_k$ is the inclusion of $A$ into $A^{\otimes n}$ as the $k$-th factor
$1 \otimes \dots \otimes A \otimes \dots \otimes 1$.
\end{prop}

\begin{proof}
First note that \eqref{swact} defines a Lie algebra action on $V^{\otimes n}\otimes \pi$. To complete the
proof it is enough to observe that this action commutes with the action of $\Sigma_n$.
\end{proof}

\begin{rem}
There are quantum analogs of this functor for a small number of variables. In the classical settings
representations of the type $A$ Hecke algebra corresponds to representations of $U_q(\gl_r)$.
In the one-variable settings this happens with the type $A$ affine Hecke algebra and the affine quantum group.
Finally for two variables there is a correspondence between double affine objects (see \cite{VV}, \cite{Gu}).
In these cases our correspondence  is the classical limit of these constructions.
\end{rem}

\begin{defn}
Let as above $A$ be a commutative finitely-generated algebra with an augmentation $\epsilon$. 
Note that $\epsilon$ can be extended as an augmentation of $S^n(A)$ in 
the natural way
$$\epsilon(a_1\circ\dots \circ a_n) = \epsilon(a_1) \dots \epsilon(a_n).$$ 
Introduce the space of
{\em diagonal harmonics}
$$DH_n(A) = \left. A^{\otimes n}\right/ S^n(A)_\epsilon \cdot A^{\otimes n}$$
as a representation of $\sdp{n}{A}$.
\end{defn}

This representation of $\Sigma_n$ is known for $A= \CC[x]$ due to the classical result of Chevalley, 
for $A=\CC[x,y]$ due to \cite{H}, for $A = \CC[x,y]/xy$ due to \cite{Ku}.
Note that Conjecture~\ref{triag} below proposes a description of this space for 
$A=\CC[x,y,z]$.

\begin{thm}\label{tdual}
For $\g = sl_r$ we have 
$$W^A_\epsilon (n\omega_1) \cong  \nabla_n^r\left(DH_n(A)\right).$$
\end{thm}

\begin{proof}
It is shown in \cite{FL2} that $W^A_\epsilon (n\omega_1)$ is the quotient of $S^n(V \otimes A)$ by 
the image of the action of $S^n(A)_\epsilon$. As we have  the natural isomorphism
$\nabla_n^r\left(A^{\otimes n} \right) \cong S^n(V \otimes A)$, it is enough to note that
the action of $S^n(A)_\epsilon$ on $S^n(V \otimes A)$ appears from its action on
$A^{\otimes n}$ by multiplication in the dual picture.
\end{proof}

Note that by taking a big enough $r$ we can reconstruct the action of $\Sigma_n$ on $DH_n(A)$
from the action of $\gl_r \otimes 1$ on $W^A_\epsilon (n\omega_1)$. And for $A = \CC[x^1, \dots, x^d]$,
$d < 4$,
the character of $W^A_\epsilon (n\omega_1)$ can be obtained from the following table.

\medskip\noindent
\begin{tabular}{|c|c|c|c|}
\hline
$d$ & $\dim W_{n\omega_1}$ & $\dim W_{n\omega_1}^{k_1\ve_1+ \dots +k_r\ve_r}$ & Reference \cr
\hline\hline
$0$ & $\bin{n+r}{r}$ & $1$   & well known \cr 
\hline
$1$ &  & 
{\bf Multinomial Coefficients}  & Corollary~\ref{gl} \cr
& $r^n$
& 
$\bin{n}{k_1, \dots k_r} = \frac{n!}{k_1!\dots k_r!}$ &  \cr
\hline
$2$ & {\bf Fuss-Catalan Number}  & {\bf ``Fuss-Narayana'' Numbers} &  Theorem~\ref{fsnr}\cr
& $ C_{n+1}^{(r)} = \left.\bin{r(n+1)}{n}\right/ (n+1)$ & 
$\left.\bin{n+1}{k_1}\bin{n+1}{k_2}\dots \bin{n+1}{k_r}\right/ (n+1)$
 &  \cr
\hline
$3$ 
& $\left.r\bin{(2r-1)(n+1)}{n-1}\right/\bin{n+1}{2}$ & $2^{r} (n+1)^{r-2} 
\prod\limits_{i=1}^r \frac{\bin{2(n+1)-k_i}{k_i}}{2(n+1)-k_i}$ & Conjecture~\ref{triag} \cr
\hline
\end{tabular}

\section{Explicit results}

\subsection{Binomial case $d=1$}

The following theorem first appeared as a conjecture and was proved for ${\mathfrak sl}_2$ in \cite{CP2},
then the proof was completed for ${\mathfrak sl}_r$ in \cite{CL}, for simply-laced Lie algebra
in \cite{FoLi} and recently H.Nakajima noticed that the general case can be obtained by
combining the results of \cite{BN}, \cite{K1}, \cite{K2}, \cite{N1}, \cite{N2}.

\begin{thm}\label{tns}\cite{CP2}\cite{CL}\cite{FoLi}\cite{N2}
We have
$$W^1 (\lambda) \cong \bigotimes_{i=1}^\rkg W^1(\omega_i)^{\otimes \lambda_i}$$
as $\g\otimes 1$-modules.
\end{thm}

\begin{rem}
Note that the structure of $\g \otimes \CC[x]$-module can be obtained by
replacing the tensor product by the {\em fusion product} from \cite{FL1}.
\end{rem}

\begin{rem}
In the one-variable case Weyl modules can be defined in the quantum group settings (see \cite{CP}) and
the analog of Theorem~\ref{tns} follows for them.
\end{rem}

\begin{crl}\label{pold1}
We have $\dim W^1(\lambda)^{\lambda-\mu}$ is a polynomial in $\lambda_1, \dots, \lambda_\rkg$ 
of degree $\mu^i$ in the variable $\lambda_i$, $i=1\dots \rkg$.
\end{crl}

\begin{proof}
If $\mu \notin Q^+$ then this dimension is constantly zero. If $\mu=0$ then it is constantly one.

We say that $\mu>\nu$ if $\mu - \nu \in Q^+$. Note that it is a partial order and
that for each $\mu \in Q^+$ there is a finite number of $\nu \in Q^+$ such that
$\mu >\nu$.
Now suppose we know the statement for all
$\nu < \mu$, and let us deduce it for the given $\mu$.

Choose $k$ such that $\lambda_k \ne 0$. Then we have 
$W^1(\lambda) \cong W^1(\lambda - \omega_k)\otimes W^1(\omega_k)$.
Note that due to the highest weight condition for each weight $\eta$ of $W^1(\omega_k)$
we have either $\eta = \omega_k$ and $\dim W^1(\omega_k)^\eta = 1$, or $\eta \le \omega_k - \alpha_k$.
So
$$\dim W^1(\lambda)^{\lambda-\mu} = 
\sum_{\nu\in Q_+} 
\left(\dim W^1(\lambda-\omega_k) ^{\lambda-\omega_k-\nu}\right)
\left(\dim W^1(\omega_k) ^{\omega_k-\mu+\nu}\right)=$$
$$=
\dim W^1(\lambda-\omega_k)^{\lambda-\omega_k-\mu} +
\sum_{{\nu\in Q_+}\atop{\nu + \alpha_k \le \mu}} 
\left(\dim W^1(\lambda-\omega_k) ^{\lambda-\omega_k-\nu}\right)
\left(\dim W^1(\omega_k) ^{\omega_k-\mu+\nu}\right).
$$
Therefore by the induction hypothesis, $\dim W^1(\lambda)^{\lambda-\mu} - \dim W^1(\lambda-\omega_k)^{\lambda-\omega_k-\mu}$
is a polynomial of degree $\mu^k -1$ in $\lambda_k$ and degree
$\mu^i$ in $\lambda_i$ for $i \ne k$. So the difference
derivatives of $\dim W^d(\lambda)^{\lambda-\mu}$ are polynomials of desired degree and the statement follows.
\end{proof}

\begin{crl}\label{gl}
For $\g = \gl_r$ we have $\dim W^1(n\omega_1) = r^n$ and 
$$\dim W^1(n\omega_1)^{k_1\ve_1+ \dots+ k_r\ve_r} = \frac{n!}{k_1! \dots k_r!}$$
when $k_1 + \dots + k_r =n$, and zero otherwise.
\end{crl}

\begin{proof}
It follows from Theorem~\ref{tns} together with the observation that
$\dim W^1(\omega_1)^\mu = 1$ for $\mu = \ve_i$, $i=1\dots r$, and zero otherwise. 
\end{proof}

\subsection{Catalan case $d=2$}

From here fix $\g = \gl_r$. Let us show that the character combinatorics of 
two-variable Weyl modules
is pretty similar to the Catalan number combinatorics.

\subsubsection{Parking representations}

Consider the following elementary problem `from a real life'. Imagine we have $n$ cars and 
exactly $n$ parking spaces along
the road. Suppose that they are distributed into a sequence of $N$ parking lots with capacities
$m_1, \dots, m_N$. Now let each car has its own preferred lot, so we have a set-theoretical
function $f: \{1,\dots, n\} \to \{1,\dots,N\}$. Now the spaces are distributed as follows.
When $i$-th car appears on the road,
it goes directly to $f(i)$-th lot, then parks there, if it is not full, or at the first available lot 
after $f(i)$ otherwise.

\begin{prop}
Each car find its space if and only if we have 
$$\left|f^{-1}(\{1,\dots, s\})\right| \ge m_1+ \dots + m_s, \qquad s=1\dots N.$$
\end{prop}

\begin{proof}
Follows by induction on the number of cars. After the first car find the space at the $k$-th lot we consider all the
other cars with the same preference function, and capacities $m_1, \dots, m_k-1, \dots, m_N$.
\end{proof}

Note that the answer does not depend on the order of the cars. It motivates the following
definition.

\begin{defn}\cite{PP}\cite{PiSt}\cite{Yan}
For a vector $\v= (m_1, \dots, m_N)$ 
set $|\v| = m_1 + \dots + m_N$ and 
introduce 
the set of {\em generalized parking functions}
$$PF(\v) = \{ f : \{1,\dots, |\v|\} \to \NN \ | \ 
\left|f^{-1}(\{1,\dots, s\})\right| \ge m_1 + \dots + m_s,\  1\le s \le N  \}$$
with the action of $\Sigma_{|\v|}$ by permutations of arguments.
\end{defn}

\begin{defn}
Introduce the representation $\cpf(\v)$ of $\Sigma_{|\v|}$ as the space with basis
$PF(\v)$ and the corresponding action of the symmetric group.
\end{defn}

\begin{rem}
It is shown in \cite{H} that  we have 
$DH_n(\CC[x,y]) \cong \cpf((n)^t)\otimes \sgr$ as representations of $\Sigma_n$, where $\sgr$ is
the one-dimensional sign representation of $\Sigma_n$. 
\end{rem}

Note that the whole set of functions from $\{1, \dots, n\}$ to $\NN$
enumerates the monomial basis of $\CC[t_1,\dots, t_n]\cong \CC[t]^{\otimes n}$,
namely each $f$ corresponds to $t_1^{f(1)-1} \dots t_n^{f(n)-1}$. So 
if $A$ acts on $\CC[t]$ not increasing the degree of polynomials
then $\sdp{|\v|}{A}$ acts on $\cpf(\v)$.
Note that for this end $A$ is not necessarily commutative.

Let $\da = \CC\left<X,Y\right>$ be the associative algebra with the relation
$XY-YX = X$. Then $\da$ acts on $\CC[t]$ as follows:
$X$ acts as $\prt$, $Y$ acts as $t\prt$. As this action does not increase the
degree, the wreath product $\sdp{|\v|}{\da}$ acts on $\cpf(\v)$.
Note that $\da$ is a PBW-algebra, namely it is a flat deformation of the commutative
algebra $\CC[x,y]$, and we expect that the Weyl modules can be deformed to
the parking function representations.

\begin{thm} \cite{FL3}
For a partition $\xi$ 
there is a filtration on $\CC\left<X,Y\right>$-module
$\nabla_n^r(\cpf(\xi^t)\otimes \sgr)$ 
whose adjoint graded factor 
is a quotient of $W^2(\xi)$.
\end{thm}

\begin{proof}
This filtration is obtained as follows: first introduce the filtration on $\CC\left<X,Y\right>$ such that
$F^m\CC\left<X,Y\right>$ is a linear span of $X^iY^j$ with $j \le m$. This filtration is Fourier dual
to the natural degree filtration (and that is why in \cite{FL3} we use a less convenient dual notation).
Then it defines a filtration on $\gl_r \otimes  \CC\left<X,Y\right>$, thus on 
$U(\gl_r \otimes  \CC\left<X,Y\right>)$ and therefore on each cyclic module. 

Let $v_1, \dots, v_r$ be the standard basis of $V$. Now introduce 
$$v_\xi = \bigwedge_{i=1}^r \bigwedge_{j=0}^{\xi_i-1} v_i\otimes t^{j} \in \nabla_n^r(\cpf(\xi^t)\otimes \sgr)$$
and note that this vector is cyclic. Set 
$$F^m \nabla_n^r(\cpf(\xi^t)\otimes \sgr) = F^m U(\gl_r \otimes  \CC\left<X,Y\right>) v_\xi,$$
then $\gr \left(\nabla_n^r(\cpf(\xi^t)\otimes \sgr)\right)$ 
is a representation of $\gr \left(\gl_r \otimes  \CC\left<X,Y\right>\right)
\cong \gl_r \otimes  \CC[x,y]$
and it is shown in \cite{FL3} that the image of $v_\xi$ is the highest weight vector. So this representation
is a quotient of the Weyl module.
\end{proof}

\begin{crl}
There is an inclusion $\nabla_n^r(\cpf(\xi^t)\otimes \sgr)$ to $W^2(\xi)$
as representations of $\gl_r \otimes 1$.
Moreover, for $\xi = (n)$ we have $W^2(\xi) \cong \nabla_n^r(\cpf(\xi^t)\otimes \sgr)$.
\end{crl}

\begin{proof}
First statement is immediate due to semisimplicity of $\gl_r$, second statement follows from \cite{H}
and Theorem~\ref{tdual}.
\end{proof}

\begin{rem}
Note that the main conjecture of \cite{FL3} that these representations are indeed isomorphic.
\end{rem}

\begin{prop}\label{xis}
We have
$\dim \nabla_n^r(\cpf(\v)\otimes \sgr)^{k_1\ve_1+ \dots+ k_r\ve_r}$ is equal to the number of $|\v|$-elements subsets 
$H \subset \{1,\dots, r\} \times \NN$ such that 
\begin{equation}\label{xigrad}
|H\cap \{i\} \times \NN | = k_i, \ \ i=1\dots r;
\end{equation}
\begin{equation}\label{xiparks}
|H\cap \{1,\dots, r\} \times \{1,\dots, s\}| \ge m_1 + \dots + m_s,\ \ 1 \le s \le N.
\end{equation}
\end{prop}

\begin{proof}
Each subset $H$ corresponds to a vector
\begin{equation}\label{xism}
v_H = \bigwedge_{i\times j \in H} v_i \otimes t^{j-1},
\end{equation}
and it is shown in \cite{FL3} that the vectors $v_H$ for $H$ under condition~\eqref{xiparks}
span the space $\nabla_n^r(\cpf(\v)\otimes \sgr)$. The condition~\eqref{xigrad} fixes the weight of
these vectors. 
\end{proof}

The following definition helps us further to work with such functions.

\begin{defn}\label{bp}
For a given $\v$
we say that $1\le s < N$ is a {\em boundary point} for $f \in  PF(\v)$ if 
$$\left|f^{-1}(\{1,\dots, s\})\right| = m_1 + \dots + m_s.$$

Similarly for a given $\v$ and a subset $H$ under condition~\eqref{xiparks} 
we say that an integer $s$ is a {\em boundary point} if
$$|H\cap \{1,\dots, r\} \times \{1,\dots, s\}| = m_1 + \dots + m_s.$$
\end{defn}

\begin{rem}
Let $\CC[x,y]_q$ be the associative algebra of $q$-polynomials, that is the algebra with relation
$xy-qyx=0$. Then $\CC[x,y]_q$ acts on $\CC[t]$ not increasing the degree, namely $x$ acts by $\prt$
and $y$ acts by $q^{t\prt}$ (that is $y\cdot t^k = q^k t^k$). 
So  $\sdp{|\v|}{\CC[x,y]_q}$ acts on $\cpf(\v)$
and $\gl_r \otimes \CC[x,y]_q$ acts on $\nabla_n^r(\cpf(\v))$.
\end{rem}

\begin{rem}
For $\xi = (n)$ the space of parking function can be quantized to a representation of the double
affine Hecke algebra (see \cite{Ch}, \cite{G}), so
Weyl modules can be quantized to a representation of the corresponding toroidal quantum group.
Unfortunately, for other $\xi$ there are no representations of these quantum algebras with
suitable characters.
\end{rem}

\subsubsection{Explicit formula for weight spaces}

The following lemma is a parking function analog of the Reni lemma about Catalan numbers (see \cite{S}).

\begin{lem}\label{Reni}
Let $f : \{1, \dots, n\} \to  \{1, \dots, n+1\}$ be an arbitrary function. Then there exists a unique cyclic permutation 
$$\sigma_k  = (1\,2\dots n+1)^k = \left(
\begin{array}{cccccccc}
1 & 2   & & \dots & & \dots & & n+1 \cr  
k & k+1 & \dots & n+1   & 1 & 2 & \dots & k-1 \cr
\end{array}\right)
$$
such that 
$$|(\sigma_k\circ f)^{-1}(\{1, \dots, s\})| \ge s \quad \mbox{for} \ 1 \le s \le n.$$ 
\end{lem}

\begin{proof}
For a given $f$ introduce $F: \{1, \dots, n+1\} \to \NN$ by
$$F(s) = |f^{-1}(\{1, \dots, s\})| - \frac{n}{n+1}s.$$
Then note that the values of $F$ are distinct, so
it has a unique minimum. Now suppose that  $F(s)$
is minimal at $s=s_0$ and  note that
the only choice to satisfy the condition is $k = n+2 - s_0$,
so the permuted numeration starts precisely after this minimum. 
\end{proof}

\begin{thm}\label{fsnr}
For $\g = \gl_r$ we have
$\dim W^2(n\omega_1) = \left.\bin{r(n+1)}{n}\right/ (n+1)$ and
$$\dim W^2(n\omega_1)^{k_1\ve_1+ \dots+ k_r\ve_r} = \frac{\bin{n+1}{k_1} \dots \bin{n+1}{k_r}}{n+1}$$
when $k_1 + \dots + k_r =n$, and zero otherwise.
\end{thm}

\begin{proof}
Note that the permutation group $\Sigma_{n+1}$ acts
on $\{1, \dots, r\} \times \{1,\dots n+1\}$ mapping $i\times j$ to $i \times \sigma(j)$.
So it acts on the subsets, preserving condition~\eqref{xigrad}.

Applying the Schur-Weyl duality to the statement of Lemma~\ref{Reni}, for any subset 
$H \subset \{1, \dots, r\} \times \{1,\dots, n+1\}$
we obtain the cyclic permutation
$\sigma_k$
such that 
$$\left|\sigma_k(H)\cap \{1, \dots, r\} \times \{1,\dots ,s\}\right| 
\ge s \quad \mbox{for} \ 1\le s \le n.$$

So the number of $n$-element subsets under condition~\eqref{xiparks} 
is the number of all 
$n$-element subsets divided by the number of cyclic permutations $n+1$, and
the number of subsets under
 conditions~\eqref{xigrad} and~\eqref{xiparks} is the number
of all subsets under condition~\eqref{xigrad} divided by the number of cyclic permutations $n+1$.
\end{proof}

\begin{rem} 
The same arguments produce a pretty  answer for $\xi = (n,\dots,n,1)$ and $\xi = (n,\dots, n,n-1)$,
but unfortunately for the more natural case $\xi =  (n,\dots,n)$ the corresponding polynomial is not reducible
over $\QQ$.
\end{rem}

\subsubsection{Recurrence relation}

Let us present a quadratic relation for characters of  Weyl modules 
that generalizes the recurrence relation for Calatan numbers.

\begin{defn}
Introduce a symmetric function $c_\xi = \ch \nabla^r_n(\cpf(\xi^t)\otimes \sgr)$, that is the Frobenius
characteristic map of $\cpf(\xi^t)\otimes \sgr$.
\end{defn}

\begin{thm}\label{rec}
\begin{enumerate}
\item
We have
$$c_{\xi} = \sum_{i=\xi_r}^{\xi_1-1}\, c_{\left(\xi_{> i}\right)'}\, \cdot \, c_{\xi_{\le i}},$$
where 
$(\xi_{> i})_j = \max(\xi_j-i,0)$, $(\xi_{\le i})_j = \min (\xi_j,i)$,
$$
(\xi')^t_j =
\begin{cases}
\xi^t_j +1 & 1 = j < \xi_1 \cr
\xi^t_j -1 & 1< j= \xi_1 \cr
 \xi^t_j & \mbox{otherwise}
\end{cases}
$$
\nopagebreak\smallskip

\begin{center}
\centerline{\psfig{file=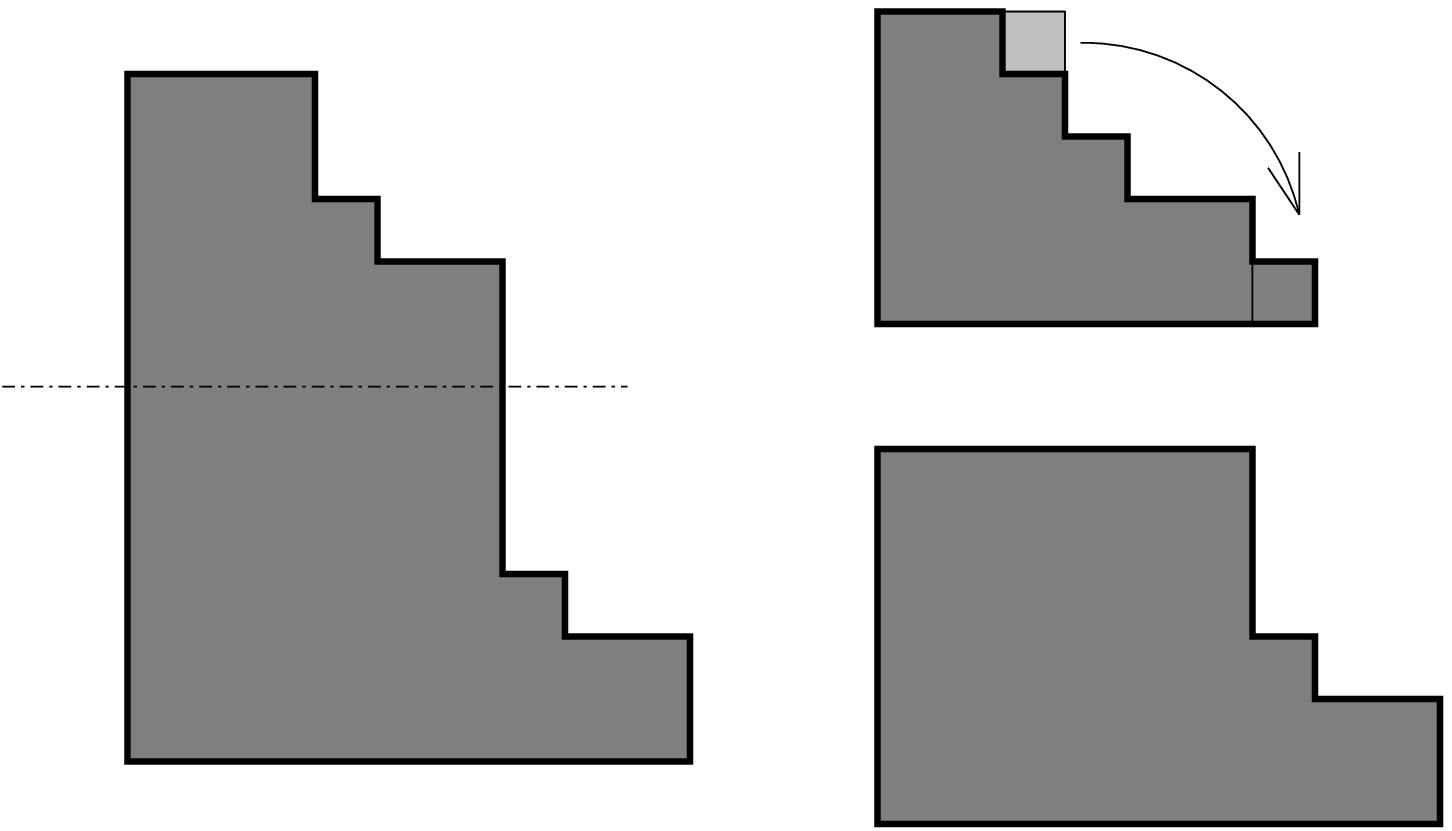,width=7cm}}
\end{center}

\item Suppose for a given $k$ we have $\xi_k - \xi_{k+1}>1$. Then 
$$c_{\xi} - c_{\xi-\alpha_k} = c_{\xi_{>(\xi_k-1)}} c_{\xi_{\le (\xi_k-1)}} + 
\sum_{i=\xi_{k+1}+1}^{\xi_k-2}  c_{(\xi_{> i}) - \alpha_k} c_{\xi_{\le i}}.$$
\end{enumerate}
\end{thm}

\begin{proof}
For (i) note that $PF((\xi')^t)\subset PF(\xi^t)$ is the subset of parking functions
without boundary points (see Definition~\ref{bp}), so the same is true in the dual picture.

Now for a subset $H$ let $i$ be the biggest boundary point.
Splitting such $H$ into two parts: 
$$H_{>i} = H\cap \{1, \dots, r\} \times \{i+1, \dots\}, \qquad 
H_{\le i}  = H\cap \{1, \dots, r\} \times \{1, \dots, i\}$$
and multiplying their contributions to character, we identify both sides of the equality.

Proof of (ii) is similar, here we notice only boundary points on the interval 
$(\xi_{k+1}, \xi_k)$ and subtracting $\alpha_k$ to avoid them.
\end{proof}

\begin{crl}\label{pol2}
We have $\dim \nabla^r_n(\cpf(\xi^t)\otimes \sgr)^{\xi-\mu}$ is a polynomial in $\xi_1, \dots, \xi_r$ 
of degree $2\mu^i$ in $\xi_i-\xi_{i+1}$.
\end{crl}

\begin{proof}
Proof is similar to the proof of Corollary~\ref{pold1}, but now we use Theorem~\ref{rec}~(i)
to show that this function is a polynomial and Theorem~\ref{rec}~(ii) to calculate the degree.

Note that the structure of Weyl module depends only on the restriction of $\xi$ to 
${\mathfrak sl}_r$.
So for $\lambda = \sum \lambda_i \omega_i$ let $P_\mu(\lambda) = \dim \nabla^r_n(\cpf(\xi^t)\otimes \sgr)^{\xi-\mu}$ with $\xi_i - \xi_{i+1} = \lambda_i$. 
Suppose that $\lambda = \lambda_k \omega_k + \dots + \lambda_{r-1}\omega_{r-1}$ with
$\lambda_k > 0$.
Then by Theorem~\ref{rec}~(i) we have
\begin{equation}\label{recwt}
P_\mu(\lambda) = 
\sum_{\nu \in Q^+} P_{\mu,\nu}(\lambda) \quad \mbox{with}
\end{equation}
\begin{align*}
P_{\mu,\nu}(\lambda) \qquad = \qquad 
P_\nu(\omega_k) &\cdot  P_{\mu-\nu}(\lambda-\omega_k) \quad + \\
+ \sum_{i=2}^{\lambda_k} P_{\nu-\alpha_k}(i\omega_k-\alpha_k)
&\cdot P_{\mu-\nu}(\lambda - i\omega_k)\quad + \\
+ \sum_{j=k+1}^{r-1} \sum_{i=1}^{\lambda_{j}}
P_{\nu-\alpha_{k\dots j}} (\lambda_{k\dots j-1} + i\omega_j - \alpha_{k\dots j})
&\cdot P_{\mu-\nu}(\lambda_{j\dots r-1}-i\omega_j),
\end{align*}
where 
$\alpha_{a\dots b} = \alpha_a + \dots + \alpha_b$, $\lambda_{a\dots b} = 
\lambda_a \omega_a + \dots + \lambda_b\omega_b$ for $1 \le a \le b <r$.

Let us perform an induction on $\mu$: suppose we know the
statement for all $\nu<\mu$.  Note that if $P_1$ and $P_2$ are polynomials
then $\sum_{i=1}^n P_1(i) P_2(n-i)$ is a polynomial in $n$. 
So all the sums in the right hand side except the first summand for
$\nu=0$ are polynomials by the induction hypothesis.
Therefore
the difference derivative
$P_\mu(\lambda) -
P_{\mu} (\lambda-\omega_k)$
is a polynomial in $\lambda_k, \dots, \lambda_{r-1}$
so by induction on $r-k$ we have $P_\mu$ is a polynomial.

To calculate the degree of $P_\mu$ in $\lambda_k$ we can now 
suppose that  $\xi_k - \xi_{k+1}= \lambda_k > 1$ and use
 Theorem~\ref{rec}~(ii)  
Here we have
\begin{equation}\label{recwt2}
P_\mu(\lambda) = P_{\mu-\alpha_k}(\lambda - \alpha_k) + \sum_{\nu \in Q^+} P_{\mu,\nu}^k(\lambda)
 \quad \mbox{with}
\end{equation}
\begin{align*}
P_{\mu,\nu}^k(\lambda) \qquad = \qquad
P_\nu(\lambda_{1\dots k-1} + \omega_k) &\cdot P_{\mu-\nu}(\lambda_{k\dots r-1} - \omega_k) \quad+\\
+ \sum_{i=2}^{\lambda_k-1} P_{\nu-\alpha_k} (\lambda_{1\dots k-1} + i\omega_k-\alpha_k)
&\cdot P_{\mu-\nu}(\lambda_{k\dots r-1} - i\omega_k).
\end{align*}
Note that if $P(i)$ is a polynomial of degree $m$ then $\sum_{i=1}^n P(i)$ is a polynomial
in $n$ of degree $m+1$. So in the case $\lambda_{1\dots k-1} = 0$ we have from~\eqref{recwt2}
by induction of $\mu$ that
$P_\mu(\lambda) - P_\mu (\lambda-\omega_k)$  has degree $2(\mu^k-1)+1$, so the statement follows.

Now general case follows from~\eqref{recwt2} by induction on $\mu$ as 
$P_\mu(\lambda) -  P_{\mu}(\lambda_{k\dots r-1} - \omega_k)$ has a smaller degree.
\end{proof}

\subsection{Hyper-Catalan case $d=3$}

\begin{cnj}\label{triag} Let $\g = \gl_r$. 
\begin{enumerate}
\item We have
$$\dim W^3(n\omega_1) = r\prod_{i=1}^{n-1} \frac{2(n+1)r - n -i}{n+2-i} = 
\left.r\bin{(2r-1)(n+1)}{n-1}\right/\bin{n+1}{2}.$$
\item We have
$$\dim W^3(n\omega_1)^{k_1\ve_1+ \dots+ k_r\ve_r} = 2^{r} (n+1)^{r-2} 
\prod_{i=1}^r \frac{\bin{2(n+1)-k_i}{k_i}}{2(n+1)-k_i}$$
when $k_1 + \dots + k_r =n$, and zero otherwise.
\end{enumerate}
\end{cnj}

\begin{rem}
This conjecture supports the conjecture of \cite{H1} that 
$$\dim DH_n(\CC[x,y,z]) = 2^n (n+1)^{n-2}.$$
Moreover, it follows that the trace of a permutation  $\sigma \in \Sigma_n$  consisting of cycles of length $k_1+1, \dots, k_m+1$
on $DH_n(\CC[x,y,z])$ is equal to
$$2^{n-k_1- \dots -k_m} (n+1)^{n-2-k_1- \dots -k_m} \prod_{i=1}^m \left[{2k_i+1}\atop{k_i+1}\right].$$

\end{rem}

\subsection{Singular cases}

First let us mention the result of \cite{Ku} on diagonal coinvariants for a double point, that is
$A = \CC[x,y]/xy\CC[x,y]$ and $\epsilon(P) = P(0,0)$.
It is shown there that for $n>0$ we have
$$W^A_\epsilon(n\omega_1) \cong V^{\otimes n} \oplus (n-1)\left(\wedge^2 V \otimes V^{\otimes n-2}\right).$$
So for $n\ge 1$ the dimensions of corresponding weight spaces are values of  a polynomial.

Another observation is related to $A=\CC[x,y]/x^l\CC[x,y]$ and $\epsilon(P) = P(0,0)$.

\begin{defn}
Introduce $PF(\v)^{(l)}$ as the subset of $PF(\v)$ formed by such functions
that on any interval $[s+1,s+l]$,
$0\le s <  N-l$, there is at least one boundary point.
\end{defn}

\begin{prop}
The quotient of $\nabla_n^r\left(\cpf(\v)\otimes \sgr\right)$
by the action of the Lie algebra $\gl_r \otimes X^l\CC\left<X,Y\right>$ is isomorphic to  
 $\nabla_n^r\left(\cpf(\v)^{(l)}\otimes \sgr\right)$.
\end{prop}

\begin{proof}
By $\CC[t]^k \subset \CC[t]$ denote the subspace of polynomials of degree less or equal $k$.
Note that $\CC\left<X,Y\right>$ acts on $\CC[t]^k$ and that the image of
$X^l\CC\left<X,Y\right>$ in ${\rm End} \left(\CC[t]^k\right)$ consists of all operators,
decreasing degree by at least $l$, so it is a linear span of  
the operators $M_{ab}$ for $a-b \ge l$
sending $t^a$ to $t^b$ and $t^i$ to zero for $i \ne a$.

By $E_{ij} \in \gl_r$ denote the matrix units. Then $E_{ij} \otimes M_{ab}$
acts on monomials~\eqref{xism} sending $v_H$ to 
$v_{H'}$ where
$$ H' = (H \setminus i\times a)\cup j \times b,$$ 
so the result contains no boundary points on the interval $[b,a-1]$.
And vise versa, if $a$ and $b-1$ are two neighbor boundary points of $H$ and $a>b$ then
there exists $i$ and $j$ such that $j \otimes b \in H$, $i \otimes a \not\in H$
and for 
$$H^\circ = (H \setminus j\times b)\cup i \times a$$
we have $v_{H^\circ} \in \nabla_n^r\left(\cpf(\v)\otimes \sgr\right)$.
So $v_H = (E_{ij} \otimes M_{ab}) v_{H^\circ}$.
\end{proof}

\begin{thm}
Let  $\g = \gl_r$  $A=\CC[x,y]/x^l\CC[x,y]$, $\epsilon(P) = P(0,0)$.
Then there exists a filtration on 
$\nabla_n^r\left(\cpf(\xi^t)^{(l)}\otimes \sgr\right)$ 
such that 
$\gr \nabla_n^r\left(\cpf(\xi^t)^{(l)}\otimes \sgr\right)$ 
is a quotient of $W^A_\epsilon(\xi)$.
\end{thm}

\begin{proof}
Note that the filtration on $\nabla_n^r\left(\cpf(\xi^t)\otimes \sgr\right)$
produces a filtration on the quotient $\nabla_n^r\left(\cpf(\xi^t)^{(l)}\otimes \sgr\right)$.
Then
$\gr \nabla_n^r\left(\cpf(\xi^t)^{(l)}\otimes \sgr\right)$ is a highest weight representation
of $\gl_r \otimes \CC[x,y]$, where $\gl_r \otimes x^l \CC[x,y]$ acts by zero, so it is 
a highest weight representation of $\gl_r \otimes A$, hence a quotient of $W^A_\epsilon(\xi)$.
\end{proof}

\begin{cnj}
We have $W^A_\epsilon(\xi)  \cong  \gr \nabla_n^r\left(\cpf(\xi^t)^{(l)}\otimes \sgr\right)$.
\end{cnj}

\begin{thm}\label{recfin}
Let $c_{\xi}^{(l)} = \ch \nabla_n^r\left(\cpf(\xi^t)^{(l)}\otimes \sgr\right)$. Then we have
$$c_{\xi}^{(l)} = \sum_{i=\max(\xi_1-l,\xi_r)}^{\xi_1-1} c_{\xi_{> i}'}^{(l)} c_{\xi_{\le i}}^{(l)}
=  \sum_{i=\max(\xi_1-l,\xi_r)}^{\xi_1-1} c_{\xi_{> i}'} c_{\xi_{\le i}}^{(l)}.$$
\end{thm}

\begin{proof}
Similar to proof of Theorem~\ref{rec} (i), just note that the last boundary point is not less than $\xi_1-l$.
Here $c_{\xi_{> i}'}^{(l)} = c_{\xi_{> i}'}$ because $\xi_{> i}'$ has $l$ rows or less.
\end{proof}

\begin{rem}
Note that there are no analog of  Theorem~\ref{rec} (ii) for this case.
\end{rem}

\begin{ex}
Set $r=2$. Let $C_n$ be the usual Catalan numbers and 
$$C_{n+1}^{(l)} = \dim \nabla_n^2\left(\cpf((n)^t)^{(l)}\otimes \sgr\right).$$
Then we have $C_n^{(l)} = C_n$ for $1\le n \le l+1$ and
$$C_{n+1}^{(l)} = 2C_{n}^{(l)} + C_{n-1}^{(l)} + 2C_{n-2}^{(l)} + 5 C_{n-3}^{(l)}
+ \dots + C_{l-1} C_{n-l+1}^{(l)}.$$
\end{ex}

\begin{crl}
Fix $\mu$ and suppose that $\xi_i - \xi_{i-1} > l\mu^i+l$ for all $i$. 
Then the dimension of $\nabla_n^r\left(\cpf(\xi^t)^{(l)}\otimes \sgr\right)^{\xi-\mu}$ is a polynomial in 
 $\xi_1, \dots, \xi_r$ of degree $\mu^i$ in $\xi_i-\xi_{i+1}$.
\end{crl}

\begin{proof}
Proof is based on Theorem~\ref{recfin} and is similar to the proof of
Corollary~\ref{pol2}.
Let $P_\mu^{(l)}(\lambda) = \dim \nabla^r_n(\cpf(\xi^t)^{(l)}\otimes \sgr)^{\xi-\mu}$ 
with $\xi_i - \xi_{i+1} = \lambda_i$. 
Then it satisfy the truncated analog of relation~\eqref{recwt} 
with only first $l$ summands in the right hand side for $P_{\mu,\nu}$.

First step is to show that $P_\mu^{(l)}(\lambda)$ 
is a polynomial in $\lambda_k$ of degree $\mu^k$
when $\lambda = \lambda_k\omega_k + \dots \lambda_{r-1}\omega_{r-1}$
and $\lambda_k > l\mu^k+l$ for $i \ge k$. Here the relation is simplified
as 
$$P_\mu^{(l)}(\lambda) = 
\sum_{\nu \in Q^+} P_{\mu,\nu}^{(l)}(\lambda) \quad \mbox{with}
$$
$$
P_{\mu,\nu}^{(l)}(\lambda) =
P_\nu^{(l)}(\omega_k) \cdot  P_{\mu-\nu}^{(l)}(\lambda-\omega_k) + 
\sum_{i=2}^{l} P_{\nu-\alpha_k}^{(l)}(i\omega_k-\alpha_k)
\cdot P_{\mu-\nu}^{(l)}(\lambda - i\omega_k),
$$
so the statement follows by induction on $\mu$.

Second step is to 
prove the same statement for an arbitrary 
$\lambda$ with $\lambda_k > l\mu^k+l$ 
by induction on $\lambda_1 + \dots + \lambda_{k-1}$ using the truncated version of~\eqref{recwt}.
\end{proof}

\end{document}